\documentclass[12pt,a4paper,reqno]{amsart}
\usepackage{amssymb}
\usepackage{verbatim}
\usepackage{amscd}
\usepackage{enumerate}
\usepackage{graphicx}
\usepackage{siunitx}
\usepackage{color}
\numberwithin{equation}{section}

\usepackage{mathabx}

\usepackage{mathtools}
\usepackage[tableposition=top]{caption}
\usepackage{booktabs,dcolumn}

\newcommand{\HH}{\mathbb{H}} 
\newcommand{\I}{\mathbb{I}} 
\newcommand{\E}{\mathbb{E}} 
\renewcommand{\P}{\mathbb{P}} 

\newcommand{\X}{\mathbf{X}} 
\newcommand{\Y}{\mathbf{Y}} 
\newcommand{\y}{\mathbf{y}} 
\newcommand{\n}{\mathbf{n}}

\theoremstyle{plain}

\newtheorem{theorem}{Theorem}[section]
\newtheorem{proposition}[theorem]{Proposition}

\newtheorem{lemma}[theorem]{Lemma}
\newtheorem{corollary}[theorem]{Corollary}
\newtheorem{conjecture}[theorem]{Conjecture}

\theoremstyle{definition}

\newtheorem{remark}[theorem]{Remark}

\newcommand\R{\mathbb{R}}
\newcommand\Z{\mathbb{Z}}
\newcommand\N{\mathbb{N}}
\newcommand\C{\mathbb{C}}

\newcommand\eps{\varepsilon}

\begin{document}

\title[Chowla and Elliott conjectures]{The logarithmically averaged Chowla and Elliott conjectures for two-point correlations}

\author{Terence Tao}
\address{Department of Mathematics, UCLA\\
405 Hilgard Ave\\
Los Angeles CA 90095\\
USA}
\email{tao@math.ucla.edu}

\begin{abstract}  Let $\lambda$ denote the Liouville function.  The Chowla conjecture, in the two-point correlation case, asserts that
$$ \sum_{n \leq x} \lambda(a_1 n + b_1) \lambda(a_2 n+b_2) = o(x) $$
as $x \to \infty$, for any fixed natural numbers $a_1,a_2$ and non-negative integer $b_1,b_2$ with $a_1b_2-a_2b_1 \neq 0$.  In this paper we establish the logarithmically averaged version
$$ \sum_{x/\omega(x) < n \leq x} \frac{\lambda(a_1 n + b_1) \lambda(a_2 n+b_2)}{n} = o(\log \omega(x)) $$
of the Chowla conjecture as $x \to \infty$, where $1 \leq \omega(x) \leq x$ is an arbitrary function of $x$ that goes to infinity as $x \to \infty$, thus breaking the ``parity barrier'' for this problem.  Our main tools are the multiplicativity of the Liouville function at small primes, a recent result of Matom\"aki, Radziwi{\l}{\l}, and the author on the averages of modulated multiplicative functions in short intervals, concentration of measure inequalities, the Hardy-Littlewood circle method combined with a restriction theorem for the primes, and a novel ``entropy decrement argument''.  Most of these ingredients are also available (in principle, at least) for the higher order correlations, with the main missing ingredient being the need to control short sums of multiplicative functions modulated by local nilsequences.

Our arguments also extend to more general bounded multiplicative functions than the Liouville function $\lambda$, leading to a logarithmically averaged version of the Elliott conjecture in the two-point case.  In a subsequent paper we will use this version of the Elliott conjecture to affirmatively settle the Erd\H{o}s discrepancy problem.
\end{abstract}

\maketitle


\section{Introduction}

Let $\lambda$ denote the Liouville function, thus $\lambda$ is the completely multiplicative function such that $\lambda(p)=-1$ for all primes $p$.  We have the following well known conjecture of Chowla \cite{chowla}:

\begin{conjecture}[Chowla conjecture]\label{chow}  Let $k \geq 1$, let $a_1,\dots,a_k$ be natural numbers and let $b_1,\dots,b_k$ be distinct nonnegative integers such that $a_i b_j - a_j b_i \neq 0$ for $1 \leq i < j \leq k$.  Then
$$ \sum_{n \leq x} \lambda(a_1 n + b_1) \dots \lambda(a_k n + b_k) = o(x) $$
as $x \to \infty$.
\end{conjecture}

Thus for instance the $k=2$ case of the Chowla conjecture implies that
\begin{equation}\label{nx}
 \sum_{n \leq x} \lambda(n) \lambda(n+1) = o(x)
\end{equation}
as $x \to \infty$.  This can be compared with the twin prime conjecture, which is equivalent to the assertion that
\begin{equation}\label{nx-2}
 \sum_{n \leq x} \theta(n) \theta(n+2) \to \infty
\end{equation}
as $x \to \infty$, where $\theta(n) := \log p$ when $n$ is equal to a prime $p$, and $\theta(n):=0$ otherwise.

The $k=1$ case of the Chowla conjecture is equivalent to the prime number theorem.  The higher $k$ cases are open, although there are a number of partial results available if one allows for some averaging in the $b_1,\dots,b_k$ parameters; see \cite{mrt}, \cite{FH} for some recent results in this direction.  The bound \eqref{nx} is equivalent to the assertion that the pairs $(\lambda(n), \lambda(n+1))$ attain each of the four sign patterns $(+1,+1)$ $(+1,-1)$, $(-1,+1)$, $(-1,-1)$ $(\frac{1}{4}+o(1))x$ times.  In \cite{hpw} it was shown that the $(+1,+1)$ and $(-1,-1)$ patterns occur at least $(\frac{1}{60}+o(1))x$ times, and the $(+1,-1)$ and $(-1,+1)$ patterns occur $\gg x \log^{-7-\eps} x$ times for $\eps > 0$.  In the recent paper \cite{mr} it was shown that in fact all four sign patterns occur $\gg x$ times, so in particular
$$
\left| \sum_{n \leq x} \lambda(n) \lambda(n+1) \right| \leq (1-\delta) x $$
for some absolute constant $\delta > 0$ and sufficiently large $x$.  An analogous claim for sign patterns $(\lambda(n),\lambda(n+1),\lambda(n+2))$ of length three was shown in \cite{mrt-2}, building upon the previous result in \cite{hil} that showed that all sign patterns of length three occur infinitely often.

The first main result of this paper is to obtain a different averaged form of the Chowla conjecture in the first nontrivial case $k=2$, in which one averages in $x$ rather than in $b_1,\dots,b_k$.  More precisely, we show

\begin{theorem}[Logarithmically averaged Chowla conjecture]\label{lach}  Let $a_1,a_2$ be natural numbers, and let $b_1,b_2$ be integers such that $a_1 b_2 - a_2 b_1 \neq 0$.  Let $1 \leq \omega(x) \leq x$ be a quantity depending on $x$ that goes to infinity as $x \to \infty$.  Then one has
\begin{equation}\label{nx-w}
 \sum_{x/\omega(x) < n \leq x} \frac{\lambda(a_1 n + b_1) \lambda(a_2 n+b_2)}{n} = o( \log \omega(x) )
\end{equation}
as $n \to \infty$.
\end{theorem}

Thus for instance this theorem implies (after setting $\omega(x) := x$, $a_1=a_2=b_2=1$ and $b_1=0$) that
\begin{equation}\label{nx-3}
 \sum_{n \leq x} \frac{\lambda(n) \lambda(n+1)}{n} = o(\log x )
\end{equation}
as $x \to \infty$; this can be deduced from \eqref{nx} by a routine summation by parts argument, but is a strictly weaker estimate.  From this and the elementary estimate $\sum_{n \leq x} \frac{\lambda(n)}{n} = o(\log x)$ we see that for any sign pattern $(\epsilon_1,\epsilon_2) \in \{-1,+1\}^2$, the set $\{n: (\lambda(n),\lambda(n+1)) = (\epsilon_1,\epsilon_2)\}$ occurs with logarithmic density $1/4$, that is to say
$$ \frac{1}{\log x} \sum_{n \leq x: (\lambda(n),\lambda(n+1))=(\epsilon_1,\epsilon_2)} \frac{1}{n} = \frac{1}{4}+o(1)$$
as $x \to \infty$.

More generally, one can deduce Theorem \ref{lach} from the $k=2$ case of Conjecture \ref{chow} by summation by parts; we leave the details to the interested reader.  Conversely, the $k=2$ case of Conjecture \ref{chow} is equivalent to the limiting case of Theorem \ref{lach} in which $\omega$ is fixed rather than going to infinity.  The logarithmic averaging is unfortunately needed in our method in order to obtain an approximate affine invariance in the $n$ variable; we do not know how to modify our argument to remove this averaging.  However, the logarithmic averaging can be tolerated in some applications (for instance to the Erd\"os discrepancy problem, discussed below).

Estimates such as \eqref{nx}, \eqref{nx-2}, \eqref{nx-w}, \eqref{nx-3} are well known to be subject to the parity problem obstruction (see e.g. \cite[Chapter 16]{opera}), and thus cannot be resolved purely by existing sieve-theoretic (or circle method) techniques that rely solely on ``linear'' estimates for the Liouville function.  We avoid the parity obstacle here by using a new ``bilinear'' estimate\footnote{Bilinear estimates have been used to get around the parity obstacle in previous works, most notably in the Friedlander-Iwaniec result \cite{fi} on primes of the form $a^2+b^4$.} for the Liouville function, which relates to bounds such as \eqref{nx-w} through the multiplicativity property $\lambda(pn) = -\lambda(n)$ of the Liouville function at small primes $p$, and which is proved using the (weak) expansion properties of a certain random graph, closely related to one recently introduced in \cite{mrt-2}.  To describe this strategy in somewhat informal terms, let us specialise to the case of establishing \eqref{nx-3} for simplicity.  Suppose for contradiction that the left-hand side of \eqref{nx-3} was large and (say) positive.  Using the multiplicativity $\lambda(pn) = -\lambda(n)$, we conclude that
$$ \sum_{n \leq x} \frac{\lambda(n) \lambda(n+p) 1_{p|n}}{n} $$
is also large and positive for all primes $p$ that are not too large; note here how the logarithmic averaging allows us to leave the constraint $n \leq x$ unchanged.  Summing in $p$, we conclude that
$$ \sum_{n \leq x} \frac{ \sum_{p \in {\mathcal P}} \lambda(n) \lambda(n+p) 1_{p|n}}{n} $$
is large and positive for any given set ${\mathcal P}$ of medium-sized primes.  By a standard averaging argument, this implies that
\begin{equation}\label{pnj}
 \frac{1}{H} \sum_{j=1}^H \sum_{p \in {\mathcal P}} \lambda(n+j) \lambda(n+p+j) 1_{p|n+j} 
\end{equation}
is large for many choices of $n$, where $H$ is a medium-sized parameter at our disposal to choose, and we take ${\mathcal P}$ to be some set of primes that are somewhat smaller than $H$.  To obtain the required contradiction, one thus wants to demonstrate significant cancellation in the expression \eqref{pnj}.  As in \cite{mrt-2}, we view $n$ as a random variable, in which case \eqref{pnj} is essentially a bilinear sum of the random sequence $(\lambda(n+1),\dots,\lambda(n+H))$ along a random graph $G_{n,H}$ on $\{1,\dots,H\}$, in which two vertices $j, j+p$ are connected if they differ by a prime $p$ in ${\mathcal P}$ that divides $n+j$.  A key difficulty in controlling this sum is that for randomly chosen $n$, the sequence $(\lambda(n+1),\dots,\lambda(n+H))$ and the graph $G_{n,H}$ need not be independent.  To get around this obstacle we introduce a new argument which we call the ``entropy decrement argument'' (in analogy with the ``density increment argument'' and ``energy increment argument'' that appear in the literature surrounding Szemer\'edi's theorem on arithmetic progressions (see e.g. \cite{tao-survey}), and also reminiscent of the ``entropy compression argument'' of Moser and Tardos \cite{moser}).  This argument, which is a simple consequence of the Shannon entropy inequalities, can be viewed as a quantitative version of the standard subadditivity argument that establishes the existence of Kolmogorov-Sinai entropy in topological dynamical systems; it allows one to select a scale parameter $H$ (in some suitable range $[H_-,H_+]$) for which the sequence $(\lambda(n+1),\dots,\lambda(n+H))$ and the graph $G_{n,H}$ exhibit some weak independence properties (or more precisely, the mutual information between the two random variables is small).  With this additional property, one can use standard concentration of measure results such as the Hoeffding inequality \cite{hoeff} to approximate \eqref{pnj} by the significantly simpler expression
$$
 \frac{1}{H} \sum_{j=1}^H \sum_{p \in {\mathcal P}} \frac{\lambda(n+j) \lambda(n+p+j)}{p}.
$$
This latter expression can then be controlled in turn by an application of the Hardy-Littlewood circle method and an estimate for short sums of a modulated Liouville function established recently by Matom\"aki, Radziwi{\l}{\l} and the author in \cite{mrt}, which is based in turn on the results of Matom\"aki and Radziwi{\l}{\l} in \cite{mr}.

The arguments in this paper extend to other bounded multiplicative functions than the Liouville function, though as they rely in an essential fashion on multiplicativity at small primes, they unfortunately do not appear to have any bearing as yet on twin prime-type sums such as \eqref{nx-2}.  More precisely, we have the following logarithmically averaged and nonasymptotic version of the Elliott conjecture \cite{elliott} (in the ``corrected'' form introduced in \cite{mrt}):

\begin{theorem}[Logarithmically averaged nonasymptotic Elliott conjecture]\label{elliott}   Let $a_1,a_2$ be natural numbers, and let $b_1,b_2$ be integers such that $a_1 b_2 - a_2 b_1 \neq 0$.   Let $\eps > 0$, and suppose that $A$ is sufficiently large depending on $\eps,a_1,a_2,b_1,b_2$.  Let $x \geq \omega \geq A$, and let $g_1,g_2\colon \N \to \C$ be multiplicative functions with $|g_1(n)|, |g_2(n)| \leq 1$ for all $n$, with $g_1$ ``non-pretentious'' in the sense that
\begin{equation}\label{tax0}
 \sum_{p \leq x} \frac{1 - \operatorname{Re} g_1(p) \overline{\chi(p)} p^{-it}}{p} \geq A 
\end{equation}
for all Dirichlet characters $\chi$ of period at most $A$, and all real numbers $t$ with $|t| \leq Ax$.  Then
\begin{equation}\label{mang} \left|\sum_{x/\omega < n \leq x} \frac{g_1(a_1 n + b_1) g_2(a_2 n + b_2)}{n}\right| \leq \eps \log \omega.
\end{equation}
\end{theorem}

\begin{remark} Our arguments are in principle effective, and would yield an explicit value of $A$ as a function of $\eps,a_1,a_2,b_1,b_2$ if one went through all the arguments carefully, however we did not do so here as we expect\footnote{For instance, a back of the envelope calculation suggests that the decay rate in the right-hand side of \eqref{nx-3} provided by optimising all the parameters in the arguments in this paper is something like $O( \frac{\log x}{(\log\log\log x)^c} )$ for some small absolute constant $c>0$; similarly, the dependence of $A$ on $1/\eps$ provided by the arguments in this paper appears to be roughly triple-exponential in nature, at least in the model case where $g_1,g_2$ are completely multiplicative and take values on the unit circle.} the bounds to be rather poor.
\end{remark}

Theorem \ref{elliott} clearly implies the following asymptotic version:

\begin{corollary}[Logarithmically averaged Elliott conjecture]\label{elliott-avg} Let $a_1,a_2$ be natural numbers, and let $b_1,b_2$ be integers such that $a_1 b_2 - a_2 b_1 \neq 0$.  Let $g_1,g_2\colon \N \to \C$ be multiplicative functions bounded in magnitude by one, with $g_1$ ``non-pretentious'' in the sense that
\begin{equation}\label{tax}
 \inf_{|t| \leq Ax} \sum_{p \leq x} \frac{1 - \operatorname{Re} g_1(p) \overline{\chi(p)} p^{-it}}{p} \to \infty
\end{equation}
as $x \to \infty$ for all Dirichlet characters $\chi$ and all $A \geq 1$.  Then for any $1 \leq \omega(x) \leq x$ which goes to infinity as $x \to \infty$, one has
\begin{equation}\label{wow}
 \sum_{x/\omega(x) < n \leq x} \frac{g_1(a_1n+b_1) g_2(a_2n+b_2)}{n} = o( \log \omega(x) )
\end{equation}
as $x \to \infty$.
\end{corollary}

\begin{remark}
If one replaced the conclusion \eqref{wow} with the stronger, non-logarithmically-averaged estimate
\begin{equation}\label{nox}
 \sum_{n \leq x} g_1(a_1n+b_1) g_2(a_2n+b_2) = o( x ),
\end{equation}
(say with $b_1,b_2 \geq 0$ to avoid the linear forms $a_1n+b_1, a_2n+b_2$ leaving the domain of $g_1,g_2$)
then this is the $k=2$ version of the corrected Elliott conjecture introduced in \cite{mrt}.  The original Elliott conjecture in \cite{elliott} replaced the condition \eqref{tax} with the weaker condition
$$ \sum_{p} \frac{1 - \operatorname{Re} g_1(p) \overline{\chi(p)} p^{-it}}{p} = +\infty$$
for all real numbers $t \in \R$, but it was shown in \cite{mrt} that this hypothesis was insufficient to establish \eqref{nox} (and it is not difficult to adapt the counterexample to also show that \eqref{wow} fails under this hypothesis).  On the other hand, it was shown in \cite{mrt} that the corrected Elliott conjecture held if one averaged in the $b_1,\dots,b_k$ parameters (rather than in the $x$ parameter as is done here).
\end{remark}

Using Vinogradov-Korobov error term zero-free region for $L$-functions (see \cite[\S 9.5]{mont}), it is not difficult to establish \eqref{tax} when $g$ is the Liouville function; see \cite[Lemma 2]{mrn} for a closely related calculation.  Thus Corollary \ref{elliott-avg} implies Theorem \ref{lach}.  Some condition of the form \eqref{tax} must be needed in order to derive the conclusion \eqref{wow}, as one can see by considering examples such as $g_1(n) := \chi(n) n^{it}$ and $g_2(n) := \overline{g_1(n)}$, where $\chi$ is a Dirichlet character of bounded conductor, $t$ is a real number of size $t = o(x)$, and $w$ is set equal to (for instance) $(x/|t|)^{1/2}$.  More precise asymptotics of sums such as those in \eqref{wow} in the ``pretentious'' case when $g_1$ and $g_2$ both behave like twisted Dirichlet characters $n \mapsto \chi(n) n^{it}$ were computed in the recent preprint of Klurman \cite{klurman}.

Corollary \ref{elliott-avg} also implies the asymptotic
$$ \sum_{n \leq x} \frac{g_1(n) g_2(n+1)}{n} = o(\log x )$$
as $x \to \infty$
when $g_1,g_2$ are multiplicative functions bounded by $1$, and at least one of $g_1,g_2$ is equal to the M\"obius function $\mu$.  Thus for instance one has
$$ \sum_{n \leq x} \frac{\mu(n) \mu(n+1)}{n}, \sum_{n \leq x} \frac{\mu^2(n) \mu(n+1)}{n}, \sum_{n \leq x} \frac{\mu(n) \mu^2(n+1)}{n}
 = o(\log x ).$$
The latter two estimates can be easily deduced from the prime number theorem in arithmetic progressions, but the first estimate is new.
 Combining this with the computations in \cite[\S 2]{mrt-2} (using logarithmic density in place of asymptotic probability), we conclude

\begin{corollary}[Sign patterns of the M\"obius function]  Let
$$ c := \prod_p \left( 1 - \frac{2}{p^2} \right) = 0.3226\dots$$
and let $(\epsilon_1,\epsilon_2) \in \{-1,0,+1\}^2$.  Then the set $\{ n: (\mu(n),\mu(n+1))=(\epsilon_1,\epsilon_2)\}$ has logarithmic density
\begin{itemize}
\item $1 - \frac{2}{\zeta(2)} + c = 0.1067\dots$ when $(\epsilon_1,\epsilon_2) = (0,0)$;
\item $\frac{1}{2} \left(\frac{1}{\zeta(2)}-c\right) = 0.1426\dots$ when $(\epsilon_1,\epsilon_2) = (+1,0), (-1,0), (0,+1), (0,-1)$; and
\item $\frac{c}{4} = 0.0806\dots$ when $(\epsilon_1,\epsilon_2) = (+1,+1), (+1,-1), (-1,+1), (-1,-1)$.
\end{itemize}
\end{corollary}

Again, the first two cases here could already be treated using the prime number theorem in arithmetic progressions, but the last case is new.
One can also use similar arguments to give an alternate proof of \cite[Theorem 1.9]{mrt-2} (that is to say, that all nine of the above sign patterns for the M\"obius function occur with positive lower density); we leave the details to the interested reader.

In a subsequent paper \cite{tao-erd}, we will combine Theorem \ref{elliott} with some arguments arising from the {\tt Polymath5} project \cite{polymath} to obtain an affirmative answer to the Erd\H{o}s discrepancy problem \cite{erdos}:

\begin{theorem}  Let $f\colon \N \to \{-1,+1\}$ be a function.  Then
$$ \sup_{d, n \in \N} \left|\sum_{j \leq n} f(jd)\right| = +\infty.$$
\end{theorem}

\subsection{Notation}

We adopt the usual asymptotic notation of $X \ll Y$, $Y \gg X$, or $X = O(Y)$ to denote the assertion that $|X| \leq CY$ for some constant $C$.  If we need $C$ to depend on an additional parameter we will denote this by subscripts, e.g. $X = O_\eps(Y)$ denotes the bound $|X| \leq C_\eps Y$ for some $C_\eps$ depending on $Y$.  Similarly, we use $X = o_{A \to \infty}(Y)$ to denote the bound $|X| \leq c(A) Y$ where $c(A)$ depends only on $A$ and goes to zero as $A \to \infty$.

If $E$ is a statement, we use $1_E$ to denote the indicator, thus $1_E=1$ when $E$ is true and $1_E=0$ when $E$ is false.

Given a finite set $S$, we use $|S|$ to denote its cardinality.

For any real number $\alpha$, we write $e(\alpha) := e^{2\pi i \alpha}$; this quantity lies in the unit circle $S^1 := \{ z \in \C: |z|=1\}$.  By abuse of notation, we can also define $e(\alpha)$ when $\alpha$ lies in the additive unit circle $\R/\Z$.

All sums and products will be over the natural numbers $\N =\{1,2,\dots\}$ unless otherwise specified, with the exception of sums and products over $p$ which is always understood to be prime.

We use $d|n$ to denote the assertion that $d$ divides $n$, and $n\ (d)$ to denote the residue class of $n$ modulo $d$.  We use $(a,b)$ to denote the greatest common divisor of $a$ and $b$.

We will frequently use probabilistic notation such as the expectation $\E \X$ of a random variable $\X$ or a probability $\P(E)$ of an event $E$; later we will also need the Shannon entropy $\HH(\X)$ of a discrete random variable, as well as related quantities such as conditional entropy $\HH(\X|\Y)$ or mutual information $\I(\X,\Y)$, the definitions of which we review in Section \ref{entropy}.  We will use boldface symbols such as $\X$, $\Y$ or $\n$ to refer to random variables.

\subsection{Acknowledgments}

The author is supported by NSF grant DMS-0649473 and by a Simons Investigator Award. The author also thanks Andrew Granville, Ben Green, Kaisa Matom\"aki, Maksym Radziwi{\l}{\l}, and Will Sawin for helpful discussions, corrections, and comments, and the anonymous referees for a careful reading of the paper and many useful suggestions and corrections.

\section{Preliminary reductions}

In this section we make a number of basic reductions, in particular reducing matters to a probabilistic problem involving a random graph, somewhat similar to one considered in \cite{mrt-2}.  Readers who are interested just in the case of the Liouville function (Theorem \ref{lach}) can skip the initial reductions and move directly\footnote{For the application to the Erd\H{o}s discrepancy problem in \cite{tao-erd}, one only needs the special case when $g_2 = \overline{g_1}$ and $g_1$ is completely multiplicative and takes values in $S^1$.  In that case one can also move directly to Theorem \ref{elliott-red}, skipping the initial reductions.} to Theorem \ref{elliott-red} below.

As mentioned in the introduction, Theorem \ref{lach} is a special case of Corollary \ref{elliott-avg}, which is in turn a corollary of Theorem \ref{elliott}.  Thus it will suffice to establish Theorem \ref{elliott}.

We first reduce to the case when $g_1$ takes values on the unit circle $S^1$:

\begin{proposition}  In order to establish Theorem \ref{elliott}, it suffices to do so in the special case where $|g_1(n)|=1$ for all $n$.
\end{proposition}

\begin{proof}  Suppose that $g_1$ takes values in the unit disk.  Then we may factorise $g_1 = g'_1 g''_1$ where $g'_1,g''_1$ are multiplicative, with $g'_1 := |g_1|$ taking values in $[0,1]$ and $g''_1$ taking values in the unit circle $S^1$.  

Let $A_0$ be a large quantity (depending on $a_1,a_2,b_1,b_2,\eps$) to be chosen later; we assume that $A$ is sufficiently large depending on $a_1,a_2,b_1,b_2,\eps,A_0$.  Suppose first that
$$
\sum_{p \leq x} \frac{1-g'_1(p)}{p} \geq A_0.
$$ 
By Mertens' theorem and the largeness of $A_0$ and $x$, this implies that
$$
\sum_{p \leq y} \frac{1-g'_1(p)}{p} \geq \frac{A_0}{2}
$$
for every $x^{1/A_0} \leq y \leq x$ (say).  Applying the Halasz inequality (see e.g. \cite{ten} or \cite[Corollary 1]{gs}) we conclude that
$$
\frac{1}{y} \sum_{n \leq y} g'_1(n) \ll A_0 \exp(-A_0/2) $$
for all $x^{1/A_0} \leq y \leq x$ (assuming $x \geq A$ and $A$ is sufficiently large depending on $A_0$).  From this and the nonnegativity and boundedness of $g'_1(n)$ it is easy to see that
$$ \sum_{x/\omega \leq n \leq x} \frac{g'_1(a_1 n+b_1)}{n} = o_{A_0 \to \infty}( \log \omega )$$
since $x \geq \omega \geq A$ and $A$ is large compared to $A_0$, and $A_0$ is large compared to $a_1,b_1$.  Since $g_1(a_1n+b_1)g_2(a_2n+b_2)$ is bounded in magnitude by $g'_1(a_1n_1+b_1)$, the claim \eqref{mang} now follows from the triangle inequality (taking $A_0$ large enough).

It remains to treat the case when
$$
\sum_{p \leq x} \frac{1-g'_1(p)}{p} < A_0.
$$ 
We now use the probabilistic method to model $g'_1$ by a multiplicative function of unit magnitude. Since $g'_1(p^j)$ takes values in the convex hull of $\{-1,+1\}$ for every prime power $p^j$, we can construct a random multiplicative function ${\mathbf g}'_1$ taking values in $\{-1,+1\}$, such that the values ${\mathbf g}'_1(p^j)$ at prime powers are jointly independent and have mean $\E {\mathbf g}'_1(p^j) = g'_1(p^j)$.  By multiplicativity and joint independence, we thus have $\E {\mathbf g}'_1(n) = g'_1(n)$ for arbitrary $n$.  By linearity of expectation we have
$$
\E \sum_{p \leq x} \frac{1-{\mathbf g}'_1(p)}{p} < A_0.
$$ 
so by Markov's inequality we see with probability $1 - O(1/A_0)$ that
$$
\sum_{p \leq x} \frac{1-{\mathbf g}'_1(p)}{p} < A_0^2.
$$
Let us restrict to this event, and set ${\mathbf g}_1 := {\mathbf g}'_1 g''_1$, thus ${\mathbf g}_1$ is a random multiplicative function taking values in $S^1$ whose mean is $g_1$.  By the triangle inequality we have
$$ {\mathbf g}_1(p) = g_1(p) + O( 1 - g'_1(p) ) + O( 1 - {\mathbf g}'_1(p) )$$
and hence by \eqref{tax0} and the triangle inequality again we have
$$ \sum_{p \leq x} \frac{1 - \operatorname{Re} \mathbf{g}_1(p) \overline{\chi(p)} p^{-it}}{p} \geq A/2$$
for all Dirichlet characters $\chi$ of period at most $A$ and all $t$ with $|t| \leq Ax$, if $A$ is large enough.  Using the hypothesis that Theorem \ref{elliott} holds when $g_1$ has unit magnitude, we conclude (again taking $A$ large enough) that
\begin{equation}\label{sa}
\left|\sum_{x/\omega < n \leq x} \frac{\mathbf{g}_1(a_1 n + b_1) g_2(a_2 n + b_2)}{n}\right| \leq \frac{\eps}{2} \log \omega
\end{equation}
with probability $1-O(1/A_0)$.  In the exceptional event that this fails, we can still bound the left-hand side of \eqref{sa} by $O( \log \omega )$.  Taking expectations, we obtain \eqref{mang} as desired (for $A_0$ large enough).
\end{proof}

A similar argument allows one to also reduce to the case where $|g_2(n)|=1$ for all $n$ (indeed, the argument is slightly simpler as \eqref{tax0} is unaffected by changes in $g_2$).

Next, we upgrade the functions $g_1,g_2$ from being multiplicative to being completely multiplicative.

\begin{proposition} In order to establish Theorem \ref{elliott}, it suffices to do so in the special case where $|g_1(n)|=|g_2(n)|=1$ for all $n$, and $g_1$ is completely multiplicative.
\end{proposition}

\begin{proof}  By the previous reductions we may already assume that $|g_1(n)|=|g_2(n)|=1$ for all $n$.  If $g_1$ is not completely multiplicative, we can introduce the completely multiplicative function $\tilde g_1$ with $\tilde g_1(p) = g_1(p)$ for all $p$.  Clearly, $\tilde g_1$ takes values in $S^1$.  From M\"obius inversion (twisted by $\tilde g_1$) we can factor $g_1$ as a Dirichlet convolution $g_1 = \tilde g_1 * h$ for a multiplicative function $h$ with $h(p)=0$ and $|h(p^j)| \leq 2$ for all $j \geq 2$; indeed we have $h(p^j) = g(p^j) - g(p) g(p^{j-1})$ for all $j \geq 1$.  The left-hand side of \eqref{mang} can then be rewritten as
$$
\left|\sum_d h(d) \sum_{x/\omega < n \leq x: d|a_1n+b_1} \frac{\tilde g_1(\frac{a_1 n + b_1}{d}) g_2(a_2 n + b_2)}{n}\right|.$$
As in the previous proposition, we choose a quantity $A_0$ that is sufficiently large depending on $a_1,a_2,b_1,b_2,\eps$, and assume $A$ is sufficiently large depending on $A_0,a_1,a_2,b_1,b_2,\eps$.  We consider first the contribution to the above sum of a single value of $d$ with $d \leq A_0$.  We crudely bound $|h(d)|$ by (say) $A_0$.  The constraint $d|a_1n+b_1$ constrains $n$ to some set of residue classes modulo $d$; the number of such classes is trivially bounded by $d$ and hence by $A_0$.  Making an appropriate change of variables and using the hypothesis that Theorem \ref{elliott} holds for completely multiplicative $g_1$ (replacing $\eps$ by $\eps/2A_0^3$, and assuming $A$ large enough), we thus have
$$ \left|\sum_{x/\omega < n \leq x: d|a_1n+b_1} \frac{\tilde g_1(\frac{a_1 n + b_1}{d}) g_2(a_2 n + b_2)}{n}\right| \leq \frac{\eps}{2A_0^2} \log \omega$$
for each $d \leq A_0$.  Thus the total contribution of those $d$ with $d \leq A_0$ is at most $\frac{\eps}{2} \log \omega$.

Now we turn to the contribution where $d > A_0$.  Here, we can use the triangle inequality to bound $\sum_{x/\omega < n \leq x: d|a_1n+b_1} \frac{\tilde g_1(\frac{a_1 n + b_1}{d}) g_2(a_2 n + b_2)}{n}$ by $O( \frac{\log \omega}{d} )$, so the net contribution of this case is $O( \log \omega \sum_{d > A_0} \frac{|h(d)|}{d} )$.  However, from taking Euler products one sees that
$$ \sum_d \frac{|h(d)|}{d^{2/3}} = O(1) $$
(say), and thus
$$ \sum_{d > A_0} \frac{|h(d)|}{d}  = O( A_0^{-1/3} ).$$
Taking $A_0$ large enough, we obtain the claim.
\end{proof}

A similar argument allows one to also reduce to the case where $g_2$ is completely multiplicative.  As $g_1,g_2$ are now multiplicative and take values in $S^1$, we have
$$ g_1(a_1n+b_1) g_2(a_2n+b_2) = \overline{g_1}(a_2) \overline{g_2}(a_1) g_1(a_1 a_2 n + a_2 b_1) g_2(a_1 a_2 n + a_1 b_2 )$$
so by replacing $a_1,a_2,b_1,b_2$ with $a_1a_2, a_1a_2, b_1a_2, b_2 a_1$ respectively, we may assume that $a_1=a_2=a$, $b_1 = b$, and $b_2 = b+h$ for some natural number $a$, integer $b$, and nonzero integer $h$.  

Finally, we observe that we can strengthen the condition $\omega \leq x$ slightly to $\omega \leq \frac{x}{\log x}$, since for $\frac{x}{\log x} < \omega \leq x$, the contribution of those $n$ for which $n \leq \log x$ can be seen to be negligible.  (Indeed, we could reduce to the case where $\omega$ grew slower than any fixed function of $x$ going to infinity, but the restriction $\omega \leq \frac{x}{\log x}$ will suffice for us, as it prevents the $n$ parameter from being extremely small.)

Putting all these reductions together, we see that Theorem \ref{elliott} will be a consequence of the following theorem.

\begin{theorem}[Logarithmically averaged nonasymptotic Elliott conjecture]\label{elliott-red}   Let $a$ be a natural number, and let $b,h$ be integers with $h \neq 0$.   Let $\eps > 0$, and suppose that $A$ is sufficiently large depending on $\eps,a,b,h$.  Let $x \geq \frac{x}{\log x} \geq \omega \geq A$, and let $g_1,g_2\colon \N \to S^1$ be completely multiplicative functions such that \eqref{tax0} holds for all Dirichlet characters $\chi$ of period at most $A$, and all real numbers $t$ with $|t| \leq Ax$.  Then
$$ \left|\sum_{x/\omega < n \leq x} \frac{g_1(a n + b) g_2(an + b+h)}{n}\right| \leq \eps \log \omega.$$
\end{theorem}

Let $a,b,h,\eps$ be as in the above theorem\footnote{The reader may initially wish to restrict to the model case $a=1, b=0, h=1$ (and also $g_1=g_2=\lambda$) in what follows to simplify the notation and arguments slightly.}.  Suppose for sake of contradiction that Theorem \ref{elliott-red} fails for this set of parameters.  By shrinking $\eps$, we may assume that $\eps$ is sufficiently small depending on $a,b,h$.  Thus for instance any quantity of the form $O_{a,b,h}(\eps)$ can be assumed to be much smaller than $1$, any quantity of the form $O_{a,b,h}(\eps^2)$ can be assumed to be much smaller than $\eps$, and so forth.  We will also need a number of large quantities, chosen in the following order\footnote{For the purposes of optimising the quantitative bounds, it seems that one should take $H_- = \exp(\eps^{-C_1})$, $H_+ = \exp(\exp(\exp(\eps^{-C_2})))$, and $A = \exp(\exp(\exp(\eps^{-C_3})))$ for some large absolute constants $C_1 < C_2 < C_3$, at least in the regime where $a,b,h$ are bounded and $\eps$ is small, and after adjusting some of the estimates below to fully optimise the bounds.}:

\begin{itemize}
\item We choose a natural number $H_-$ that is sufficiently large depending on $a,b,h,\eps$.
\item Then, we choose a natural number $H_+$ that is sufficiently large depending on $H_-,a,b,h,\eps$.
\item Finally, we choose a quantity $A>0$ that is sufficiently large depending on $H_+,H_-,a,b,h,\eps$.
\end{itemize}
The quantity $A$ is of course the one we will use in Theorem \ref{elliott-red}.  The intermediate parameters $H_-,H_+$ will be the lower and upper ranges for a certain medium-sized scale $H \in [H_-,H_+]$ which we will later select using a pigeonholing argument which we call the ``entropy decrement argument''.

We will implicitly take repeated advantage of the above relative size assumptions between the parameters $A,H_+,H_-,a,b,h,\eps$ in the sequel to simplify the estimates; in particular, we will repeatedly absorb lower order error terms into higher order error terms when the latter would dominate the former under the above assumptions.  Thus for instance $O_{H_+,H_-,a,b,h,\eps}(1) \times o_{A \to \infty}(1)$ can be simplified to just $o_{A \to \infty}(1)$ by the assumption that $A$ is sufficiently large depending on all previous parameters, and $o_{A \to \infty}(1) + o_{H_- \to \infty}(1)$ can similarly be simplified to $o_{H_- \to \infty}(1)$.
The reader may wish to keep the hierarchy
$$ a,b,h \ll \frac{1}{\eps} \ll H_- \ll p \ll H \ll H_+ \ll A \leq \omega \leq \frac{x}{\log x} \leq x$$
and also
$$ x \geq n \geq x/\omega \geq \log x \geq \log A \gg H_+$$
in mind in the arguments that follow.

As we are assuming that Theorem \ref{elliott-red} fails for the indicated choice of parameters, there exist real numbers
\begin{equation}\label{abig}
 x \geq \omega \geq A
\end{equation}
and completely multiplicative functions $g_1, g_2\colon \N \to S^1$  such that
\begin{equation}\label{tax0-A}
 \sum_{p \leq x} \frac{1 - \operatorname{Re} g_1(p) \overline{\chi(p)} p^{-it}}{p} \geq A 
\end{equation}
for all Dirichlet characters $\chi$ of period at most $A$, and all real numbers $t$ with $|t| \leq Ax$, but such that
\begin{equation}\label{go} 
\left|\sum_{x/\omega < n \leq x} \frac{g_1(an+b) g_2(an+b+h)}{n}\right| > \eps \log \omega.
\end{equation}

To use the hypothesis \eqref{tax0-A}, we apply the results in \cite{mrt} to control short sums of $g_1$ modulated by Fourier characters.

\begin{proposition}\label{mrp}  Let the notation and assumptions be as above.  For all $H_- \leq H \leq H_+$, one has
\begin{equation}\label{gch-0}
\sup_\alpha \sum_{x/\omega < n \leq x} \frac{1}{Hn} \left|\sum_{j=1}^H g_1(n+j) e(j \alpha)\right| \ll \frac{\log\log H}{\log H} \log \omega.
\end{equation}
In particular, one has
\begin{equation}\label{gch}
\sup_\alpha \sum_{x/\omega < n \leq x} \frac{1}{Hn} \left|\sum_{j=1}^H g_1(n+j) e(j \alpha)\right| = o_{H_- \to \infty}( \log \omega ).
\end{equation}
\end{proposition}

We remark that Proposition \ref{mrp} is the \emph{only} way in which we will take advantage of the hypothesis \eqref{tax0-A}, which may now be discarded in the arguments that follow.

\begin{proof}  Let $\alpha \in \R$.  Applying \cite[Lemma 2.2, Theorem 2.3]{mrt} (with $W := \log^5 H$), we see that 
\begin{align*}
\frac{1}{X} \sum_{X \leq n \leq 2X} \left|\frac{1}{H} \sum_{j=1}^H g_1(n+j) e(\alpha j)\right| &\ll \frac{\log\log H}{\log H}
\end{align*}
for all $\frac{x}{2\omega} \leq X \leq 2x$; for the purposes of verifying the hypotheses in \cite{mrt}, we note that $X \geq \frac{x}{2\omega} \geq \frac{\log x}{2} \geq \frac{\log A}{2}$,
and hence $W = \log^5 H$ will be much less than $A$ or $(\log X)^{1/125}$.   Averaging this estimate from $X$ between $x/2\omega$ and $2x$, we obtain \eqref{gch-0} and hence \eqref{gch}.
\end{proof}

It will be convenient to interpret these estimates in probabilistic language (particularly when we start using the concept of Shannon entropy in the next section).
We introduce a (discrete) random variable $\n$ in the interval $\{ n \in \N: x/\omega < n \leq x\}$ by setting
$$ \P( \n = n ) = \frac{1/n}{\sum_{n \in \N: x/\omega < n \leq x} \frac{1}{n}}$$
whenever $n$ lies in this interval.

From \eqref{abig} and our hypothesis $\omega \leq x/\log x$,  we see that
$$\sum_{n \in \N: x/\omega < n \leq x} \frac{1}{n} = (1 + o_{A \to \infty}(1)) \log \omega.$$
We conclude from \eqref{go} that
\begin{equation}\label{face}
 | \E g_1(a\n+b) g_2(a\n+b+h) | \gg \eps 
\end{equation}
while from \eqref{gch} we conclude that
\begin{equation}\label{hq}
\sup_\alpha \E  \left|\sum_{j=1}^H g_1(\n+j) e(\alpha j)\right| = o_{H_- \to \infty}(H) 
\end{equation}
uniformly for all $H_- \leq H \leq H_+$.

The logarithmic averaging in the $n$ variable gives an approximate affine invariance to these probabilities and expectations (cf. \cite[Lemma 2.3]{mrt-2}), which is of fundamental importance to our approach:

\begin{lemma}[Approximate affine invariance]\label{linear}  Let $q$ be a natural number bounded by $H_+$, and let $r$ be a fixed integer with $|r| \leq H_+$.  Then for any event $P(\n)$ depending on $\n$, one has
$$ \P( P(\n) \hbox{ and } \n = r \ (q) ) = \frac{1}{q} \P( P( q\n+r ) ) + o_{A \to \infty}(1).$$
More generally, for any complex-valued random variable $X(\n)$ depending on $\n$ and bounded in magnitude by $O(1)$, one has
$$ \E( X(\n) 1_{\n = r \ (q)} ) = \frac{1}{q} \E( X(q\n+r) ) + o_{A \to \infty}(1).$$
\end{lemma}

Note in particular that this lemma implies the approximate translation invariance $\P( P(\n+r)) = \P(P(\n))+o_{A \to \infty}(1)$ and $\E( X(\n+r) ) = \E( X(\n) ) + o_{A \to \infty}(1)$ for any $r=O(H_+)$.  If we did not perform a logarithmic averaging, then we would still have approximate translation invariance, but we would not necessarily have the more general approximate \emph{affine} invariance, which causes the remainder of our arguments to break down.

\begin{proof} It suffices to prove the latter claim.  The left-hand side can be written as
$$ \frac{1+o_{A \to \infty}(1)}{\log \omega} \sum_{x/\omega < n \leq x: n = r \ (q)} \frac{X( n )}{n}.$$
Making the change of variables $n = qn' + r$, noting that $\frac{1}{n}$ is equal to $\frac{1}{q} \frac{1}{n'} + o_{A \to \infty}(\frac{1}{n'})$ uniformly in $n'$, we can write the previous expression as
$$ \frac{1+o_{A \to \infty}(1)}{\log \omega} \sum_{x/\omega < qn'+r \leq x} \left(\frac{1}{q} \frac{X( qn'+r)}{n'} + o_{A \to \infty}\left( \frac{1}{n'}\right)\right).$$
The net contribution of the $o_{A \to \infty}( \frac{1}{n'})$ term can be seen to be $o_{A \to \infty}(1)$ (recall that $A$ is assumed large compared to $H_+$ and hence with $q$).  The constraint $x/\omega < qn'+r \leq x$ can be replaced with $x/\omega < n' \leq x$ while incurring an error of $O( \frac{1+o_{A \to \infty}(1)}{\log \omega} O(\log q)) = o_{A \to \infty}(1)$.  The claim follows.
\end{proof}

We now give a simple application of the above lemma.
By Fourier expansion (or by positivity) we may insert the constraint $1_{a|\n}$ in the left-hand side of \eqref{hq} (recalling that $H_-$ is assumed sufficiently large depending on $a$), and thus by Lemma \ref{linear} we also have
\begin{equation}\label{hq-a}
\sup_\alpha \E  \left|\sum_{j=1}^H g_1(a\n+j) e(\alpha j)\right| = o_{H_- \to \infty}(H). 
\end{equation}
This estimate will be useful later in the argument.

From Lemma \ref{linear} and \eqref{face} we have
\begin{equation}\label{face-2}
 |\E 1_{\n=b \ (a)} g_1(\n) g_2(\n+h)| \gg \eps.
\end{equation}
Crucially, we can exploit the multiplicativity of $g_1,g_2$ at medium-sized primes to average this lower bound by further application of Lemma \ref{linear}:

\begin{proposition}\label{conv}  Assume that the bound \eqref{face-2} holds.  Let $H_- \leq H \leq H_+$.  Let ${\mathcal P}_H$ denote the set of primes between $\frac{\eps^2}{2} H$ and $\eps^2 H$. For each prime $p$, let $c_p \in S^1$ denote the coefficient $c_p := \overline{g_1}(p) \overline{g_2}(p)$.  Then one has
\begin{equation}\label{hph}
\left| \E \sum_{p \in {\mathcal P}_H} \sum_{j: j,j+ph \in [1,H]} c_p 1_{a\n+j = pb \ (ap)} g_1(a\n+j) g_2(a\n+j+ph) \right| \gg \eps \frac{H}{\log H}.
\end{equation}
\end{proposition}

We remark that in the Liouville case $g_1=g_2=\lambda$ (and also in the case $g_2 = \overline{g_1}$ required in the Erd\H{o}s discrepancy problem application in \cite{tao-erd}), we have $c_p=1$ for all $p$.  This leads to some minor simplification in the arguments (in particular, we only need to apply Proposition \ref{mrp} for ``major arc'' values of $\alpha$, allowing one to replace \cite[Lemma 2.2, Theorem 2.3]{mrt} by the simpler \cite[Theorem A.1]{mrt}), however it turns out that existing results in the literature (in particular, the restriction theorem for the primes in \cite{gt-selberg}) allow us to handle the extension to more general $c_p$ without much additional difficulty.

A key point here is that Proposition \ref{conv} applies for \emph{all} scales $H$ in the range $[H_-,H_+]$.  This is because we will not be able to compute the left-hand side of \eqref{hph} for any specified $H$; however, the ``entropy decrement argument'' we will use in the next section will locate (basically thanks to the pigeonhole principle) a single scale $H$ in the range $[H_-,H_+]$ for which the left-hand side of \eqref{hph} can be evaluated, at which point we can apply the above proposition.  The inability to specify the scale $H$ in advance is a key reason why we were unable to remove the logarithmic averaging from our final result in Theorem \ref{elliott}.

\begin{proof}  Write
$$ X :=  \E 1_{\n=b \ (a)} g_1(\n) g_2(\n+h), $$
thus \eqref{face-2} tells us that $|X| \gg \eps$.  From complete multiplicativity and the definition of $c_p$ we see that
$$
1_{\n=b \ (a)} g_1(\n) g_2(\n+h) = c_p 1_{p\n = pb \ (ap)} g_1(p\n) g_2(p\n+ph) $$
and thus
\begin{equation}\label{toc-2}
\E c_p 1_{p\n = pb \ (ap)} g_1(p\n) g_2(p\n+ph) = X
\end{equation}
for any $p \in {\mathcal P}_H$.
We now claim that
\begin{equation}\label{toc}
 \E c_p 1_{\n+j = pb \ (ap)} g_1(\n+j) g_2(\n+j+ph) = \frac{1}{p} X + o_{A \to \infty}(1)
\end{equation}
for any $1 \leq j \leq H$ and any $p \in {\mathcal P}_H$.  To see this, we split $1_{\n+j=pb\ (ap)}$ as $1_{\n = -j\ (p)} 1_{\n+j = pb\ (a)}$ and apply Lemma \ref{linear} to write the left-hand side of \eqref{toc} as
$$ \frac{1}{p} \E c_p 1_{p\n = pb\ (a)} g_1(p\n) g_2(p\n+ph) + o_{A \to \infty}(1);$$
since $1_{p\n = pb\ (a)} = 1_{p\n = pb\ (ap)}$, the claim now follows from \eqref{toc-2}.

Summing \eqref{toc} over $j=1,\dots,H$, we have
\begin{equation}\label{toc-3}
 \E c_p \sum_{j=1}^H 1_{\n+j = pb \ (ap)} g_1(\n+j) g_2(\n+j+ph) = \frac{1}{p} H X + o_{A \to \infty}(1).
\end{equation}
Now let us introduce the quantities
\begin{equation}\label{ssa}
Q(s) := \E c_p \sum_{j=1}^H 1_{\n+j = pb \ (ap)} g_1(\n+j) g_2(\n+j+ph) 1_{\n = s \ (a)} 
\end{equation}
for $s \in \Z/a\Z$.  From \eqref{toc-3} we have
\begin{equation}\label{judo}
\sum_{s \in \Z/a\Z} Q(s) =  \frac{1}{p} H X + o_{A \to \infty}(1).
\end{equation}

Now let us compare $Q(s)$ with $Q(s+1)$.  Using Lemma \ref{linear} to replace $\n$ with $\n+1$, we see that
\begin{align*}
Q(s+1) &=  \E c_p \sum_{j=1}^H 1_{\n+1+j = pb \ (ap)} g_1(\n+1+j) g_2(\n+1+j+ph) 1_{\n+1 = s+1 \ (a)} + o_{A \to \infty}(1)\\
&=  \E c_p \sum_{j=2}^{H+1} 1_{\n+j = pb \ (ap)} g_1(\n+j) g_2(\n+j+ph) 1_{\n = s \ (a)} + o_{A \to \infty}(1).
\end{align*}
Note that the difference between $ \sum_{j=2}^{H+1} 1_{\n+j = pb \ (ap)} g_1(\n+j) g_2(\n+j+ph)$ and $\sum_{j=1}^{H} 1_{\n+j = pb \ (ap)} g_1(\n+j) g_2(\n+j+ph)$ is zero with probability $1-O(1/p)$, and is $O(1)$ in the remaining event.  Absorbing the $o_{A \to \infty}(1)$ error in the $O(1/p)$ error, we conclude that
$$ Q(s+1) = Q(s) + O(1/p)$$
for all $s \in \Z/a\Z$.  Thus $Q$ fluctuates by at most $O(a/p)$, and in particular
$$ Q(0) = \frac{1}{a} \sum_{s \in \Z/a\Z} Q(s) + O(a/p).$$
Combining this with \eqref{judo}, we conclude that
$$ \E c_p \sum_{j=1}^H 1_{\n+j = pb \ (ap)} g_1(\n+j) g_2(\n+j+ph) 1_{\n = 0 \ (a)} = \frac{H X}{ap}  + O\left(\frac{a}{p}\right).$$
Summing over ${\mathcal P}_H$, we conclude that
$$
 \E \sum_{j=1}^H \sum_{p \in {\mathcal P}_H} c_p 1_{\n+j = pb \ (ap)} g_1(\n+j) g_2(\n+j+ph) 1_{\n=0 \ (a)} = \left(\frac{HX}{a} + O(a)\right) \sum_{p \in {\mathcal P}_H} \frac{1}{p} 
$$
and hence by the prime number theorem and the lower bound $|X| \gg \eps$, one has
$$
\left| \E \sum_{j=1}^H \sum_{p \in {\mathcal P}_H} c_p 1_{\n+j = pb \ (ap)} g_1(\n+j) g_2(\n+j+ph) 1_{\n=0 \ (a)} \right| \gg \eps \frac{H}{a \log H}.$$
Applying Lemma \ref{linear}, we obtain 
$$
\left| \E \sum_{j=1}^H \sum_{p \in {\mathcal P}_H} c_p 1_{a\n+j = pb \ (ap)} g_1(a\n+j) g_2(a\n+j+ph) \right| \gg \eps \frac{H}{\log H}.$$
If $j+ph$ lies outside of the interval $[1,H]$, then $j$ lies in either $[1,|h| \eps^2 H]$ or $[(1-|h| \eps^2) H, H]$.  The contribution of these values of $j$ can be easily estimated to be $O( \sum_{p \in {\mathcal P}_H} \frac{|h| \eps^2 H}{p}) = O_h( \eps^2 \frac{H}{\log H} )$, so from the smallness of $\eps$ we may discard these intervals and conclude the claim.
\end{proof}

We will shortly need to deploy the theory of Shannon entropy, at which point we encounter the inconvenient fact that $g$ could potentially take an infinite number of values and thus have unbounded Shannon entropy.  To get around this, we perform a standard discretisation.  Namely, define $g_{i,\eps^2}(n)$ for $i=1,2$ to be $g_i(n)$ rounded to the nearest element of the lattice $\eps^2 \Z[i]$, where $\Z[i]$ denotes the Gaussian integers.  (We break ties arbitrarily.) This function is no longer multiplicative, but it takes at most $O_\eps(1)$ values, it is bounded in magnitude by $O(1)$, and we have $g_{i,\eps^2} = g_i + O(\eps^2)$ for $i=1,2$.  Thus from the above proposition and the triangle inequality, we have
$$\left| \E \sum_{p \in {\mathcal P}_H} c_p \sum_{j: j, j+ph \in [1,H]} 1_{a\n+j = pb \ (ap)} g_{1,\eps^2}(a\n+j) g_{2,\eps^2}(a\n+j+ph) \right| \gg \eps \frac{H}{\log H}$$
since the error incurred by replacing $g_i$ with $g_{i,\eps^2}$ can be computed to be $O_a( \eps^2 \sum_{p \in {\mathcal P}_H} \frac{H}{p} ) = O_a( \eps^2 \frac{H}{\log H} )$.
We rewrite this inequality as
\begin{equation}\label{efy-large}
 |\E F(\X_H, \Y_H) | \gg \eps \frac{H}{\log H}
\end{equation}
where $\X_H$ is the discrete random variable
$$ \X_H := (g_{i,\eps^2}(a\n+j))_{i=1,2; j=1,\dots,H}$$
(taking values in $(\eps^2 \Z[i])^{2H}$),
$\Y_H$ is the random variable
$$ \Y_H := \n \ (P_H)$$
(taking values in $\Z/P_H\Z$) where $P_H:= \prod_{p \in {\mathcal P}_H} p$, and $F \colon (\eps^2 \Z[i])^{2H} \times \Z/P_H\Z \to \C$ is the function
\begin{equation}\label{faxy}
F( (x_{i,j})_{i=1,2; j=1,\dots,H}, y \ (P_H) ) := \sum_{p \in {\mathcal P}_H} c_p \sum_{j: j,j+ph \in [1,H]} 1_{ay+j = pb \ (ap)} x_{1,j} x_{2,j+ph}.
\end{equation}
(Note that the residue class $ay \ (ap)$ is well defined for $p \in \Z/P_H \Z$ and $p \in {\mathcal P}_H$, noting that $P_H$ is coprime to $a$.)

It is thus of interest to try to calculate the typical value of $F(\X_H,\Y_H)$.  One can interpret $F(\X_H,\Y_H)$ as a ``bilinear'' expression of the components of $\X_H$ along a certain random graph determined by $\Y_H$.  A key difficulty is that the random variables $\X_H$ and $\Y_H$ are not independent, and could potentially be coupled together in an adversarial fashion.  In this worst case, this would require one to establish a suitable ``expander'' property for the random graph associated to $\Y_H$ that would ensure cancellation in the sum regardless of what values that $\X_H$ will take.  It may well be that such an expansion property\footnote{Actually, to be able to plausibly expect expansion, one should enlarge ${\mathcal P}_H$ to be something like the primes between $H^\delta$ and $\eps^2 H$ for some small $\delta$, so that the average degree of the random graph associated to $\Y_H$ is significantly larger than one.} holds (with high probability, of course). However, we can avoid having to establish such a strong expansion property by taking advantage of an ``entropy decrement argument'' to give some weak independence between $\X_H$ and $\Y_H$ for at least one choice of $H$ between $H_-$ and $H_+$.  Once one obtains such a weak independence, it turns out that one only needs to show that for a typical choice of $\X_H$, that $F(\X_H,\Y_H)$ is small for \emph{most} choices of $\Y_H$, where we allow a (nearly) exponentially small failure set for the $\Y_H$.  This turns out to be much easier to establish than the expander graph property, being obtainable from standard concentration of measure inequalities (such as Hoeffding's inequality), and an application of the Hardy-Littlewood circle method.  

\begin{remark}
The entropy decrement argument we give below can be viewed as a quantitative variant of the construction of the Kolmogorov-Sinai entropy of a topological dynamical system (see e.g. \cite{billingsley}), but we will not explicitly use the language of topological dynamics here.  See however \cite{ab} for a discussion of the Chowla conjecture and its relation to a conjecture of Sarnak \cite{sarnak} from a topological dynamics point of view.  It may well be that the arguments here could also benefit from a more explicit use of topological dynamics machinery.
\end{remark}

\section{The entropy decrement argument}\label{entropy}

We continue the proof of Theorem \ref{elliott}.  We begin by briefly reviewing the basic Shannon inequalities from information theory.

Recall that if $\X$ is a discrete random variable (taking at most countably many values), the \emph{Shannon entropy} $\HH(\X)$ is defined\footnote{In the information theory literature, the logarithm to base $2$ is often used to define entropy, rather than the natural logarithm, in which case $\HH(\X)$ can be interpreted as the number of bits needed to describe $\X$ on the average.  One could use this choice of base in the arguments below if desired, but ultimately the choice of base is a normalisation which has no impact on the final bounds.} by the formula
$$ \HH(\X) := \sum_x \P( \X = x ) \log \frac{1}{\P(\X=x)}$$
where $x$ takes values in the essential range of $\X$ (that is to say, those $x$ for which $\P(\X=x)$ is nonzero).  A standard computation then gives the identity
\begin{equation}\label{haxy}
 \HH(\X,\Y) = \HH(\X|\Y) + \HH(\Y) = \HH(\X) + \HH(\Y|\X)
\end{equation}
for the joint entropy $\HH(\X,\Y)$ of the random variable $(\X,\Y)$,
where the \emph{conditional entropy} $\HH(\X|\Y)$ is defined by the formulae
\begin{equation}\label{xy}
\HH(\X|\Y) := \sum_y \P( \Y=y) \HH( \X|\Y=y )
\end{equation}
(with $y$ ranging over the essential range of $\Y$) and
$$ \HH(\X|\Y=y) := \sum_x \P( \X = x | \Y=y ) \log \frac{1}{\P(\X=x | \Y=y)} $$
with $\P(E|F) := \P( E \wedge F ) / \P(F)$ being the conditional probability of $E$ relative to $F$, and the sum is over the essential range of $\X$ conditioned to $\Y=y$.  From the concavity of the function $x \mapsto x \log \frac{1}{x}$ and Jensen's inequality we have
\begin{equation}\label{hyx}
 \HH(\X|\Y) \leq \HH(\X)
\end{equation}
so we conclude the subadditivity of entropy
\begin{equation}\label{subadd}
 \HH(\X,\Y) \leq \HH(\X)+\HH(\Y).
\end{equation}
If we define the \emph{mutual information}
\begin{equation}\label{mutual}
\I(\X,\Y) := \HH(\X)+\HH(\Y) - \HH(\X,\Y) = \HH(\X) - \HH(\X|\Y) = \HH(\Y) - \HH(\Y|\X)
\end{equation}
between two discrete random variables $\X,\Y$, we thus see that $\I(\X,\Y) = \I(\Y,\X) \geq 0$.

\begin{remark} One can view $\I(\X,\Y)$ as a measure of the extent to which the random variables $\X,\Y$ are not independent.  For instance, one can show that $\I(\X,\Y)=0$ if and only if $\X$ and $\Y$ are jointly independent.  In a similar vein, one can view the conditional entropy $\HH(\X|\Y)$ as a measure of the amount of new information carried by $\X$, given that one already knows the value of $\Y$.
\end{remark}

Conditioning the random variables $\X,\Y$ to an auxiliary discrete random variable $\mathbf{Z}$, we conclude the relative subadditivity of entropy
\begin{equation}\label{subadd-rel}
 \HH(\X,\Y|\mathbf{Z}) \leq \HH(\X|\mathbf{Z})+\HH(\Y|\mathbf{Z}).
\end{equation}
Finally, a further application of Jensen's inequality gives the bound
\begin{equation}\label{jens}
\HH(\X) \leq \log N
\end{equation}
whenever $\X$ takes on at most $N$ values.

Recall the discrete random variables $\X_H, \Y_H$ defined previously.  From \eqref{jens}, \eqref{subadd}, and the fact that each component of $\X_H$ takes on only $O_\eps(1)$ values, we have the upper bound
\begin{equation}\label{bad}
0 \leq \HH( \X_H ) \ll_\eps H.
\end{equation}
Note that $\Y_H$ is within $o_{A \to \infty}(1)$ (in any reasonable metric) of being uniformly distributed on $\Z/P_H\Z$, thus
\begin{equation}\label{hayah}
 \HH(\Y_H) = \log P_H - o_{A \to \infty}(1).
\end{equation}
In particular, from the prime number theorem we have the crude bound
\begin{equation}\label{hyh}
 \HH(\Y_H) \ll H 
\end{equation}
for all $H_- \leq H \leq H_+$.

Let us temporarily define the variant
$$ \X_{H_1,H_1+H_2} := \HH( (g_{i,\eps^2}(\n+j))_{i=1,2; j=H_1+1,\dots,H_1+H_2} $$
of $\X_H$, where $H_1,H_2$ are natural numbers.
From the approximate translation invariance provided by Lemma \ref{linear}, we see that
$$ \HH( \X_{H_1,H_1+H_2} ) =  
\HH( \X_{H_2} ) + o_{A \to \infty}(1)$$
for any $H_1,H_2 \leq H_+$; applying \eqref{subadd}, and noting that $\X_{H_1+H_2}$ is the concatenation of $\X_{H_1}$ and $\X_{H_1,H_1+H_2}$, we obtain the approximate subadditivity property
\begin{equation}\label{xh1}
 \HH( \X_{H_1+H_2} ) \leq \HH(\X_{H_1}) + \HH(\X_{H_2}) + o_{A \to \infty}(1)
\end{equation}
for any natural numbers $H_1,H_2 \leq H_+$.

We can improve this inequality if $\X_H$ shares some mutual information with $\Y_H$, as $\Y_H$ does not generate any entropy upon translation.  Indeed, 
from Lemma \ref{linear} again, we see for any natural numbers $H,H_1,H_2$ between $H_-$ and $H_+$ that
$$ \HH( \X_{H_1,H_1+H_2} | \n+H_1 \ (P_H) ) =  
\HH( \X_{H_2} | \n \ (P_H) ) + o_{A \to \infty}(1).$$
But $\n+H_1 \ (P_H)$ conveys exactly the same information as $\n \ (P_H)$ (they generate exactly the same finite $\sigma$-algebra of events), so
$$ \HH( \X_{H_1,H_1+H_2} | \n+H_1 \ (P_H) ) =  
\HH( \X_{H_1,H_1+H_2}| \n \ (P_H) ).$$
Inserting these identities into \eqref{subadd-rel} and recalling that $\Y_H = \n\ (P_H)$, we obtain the relative approximate subadditivity property
$$ \HH( \X_{H_1+H_2} | \Y_H ) \leq \HH(\X_{H_1} | \Y_H ) + \HH(\X_{H_2} | \Y_H) + o_{A \to \infty}(1)$$
for any $H,H_1,H_2$ between $H_-$ and $H_+$.  Iterating this, we conclude in particular that
$$ \HH( \X_{kH} | \Y_H ) \leq k \HH( \X_H | \Y_H ) + o_{A \to \infty}(1)$$
for any natural numbers $k,H$ with $H_- \leq H \leq kH \leq H_+$ (note that the number of iterations here is at most $H_+$, so that the $o_{A \to \infty}(1)$ error stays under control).  From this and \eqref{haxy},  \eqref{mutual} we see that
\begin{align*}
\HH(\X_{kH}) &= \HH(\X_{kH}|\Y_H) + \HH(\Y_H) - \HH( \Y_H | \X_{kH} ) \\
&\leq \HH(\X_{kH}|\Y_H) + \HH(\Y_H) \\
&\leq k \HH(\X_H | \Y_H) + \HH(\Y_H) + o_{A \to \infty}(1)\\
&= k \HH(\X_H)  - k \I(\X_H,\Y_H) +  \HH(\Y_H) + o_{A \to \infty}(1)
\end{align*}
which on dividing by $kH$ and using \eqref{hyh} gives
\begin{equation}\label{splat}
\frac{\HH(\X_{kH})}{kH} \leq \frac{\HH(\X_H)}{H} - \frac{\I(\X_H,\Y_H)}{H} + O\left( \frac{1}{k} \right),
\end{equation}
whenever $H_- \leq H \leq kH \leq H_+$ (note that we can absorb the $o_{A \to \infty}(1)$ error in the $O( 1/k)$ term since $k \leq H_+$).  This can be compared with the inequality 
$$ \frac{\HH(\X_{kH})}{kH} \leq \frac{\HH(\X_H)}{H} + o_{A \to \infty}(1)$$
under the same hypotheses on $H,k$, coming from iterating \eqref{xh1}.  Thus we see that the presence of mutual information between $\X_H$ and $\Y_H$ causes a decrement in the entropy rate of $\X_H$ as one increases $H$.

We can iterate this inequality and use an ``entropy decrement argument'' to get a non-trivial upper bound on the mutual information $\I(\X_H,\Y_H)$ for some large $H$:

\begin{lemma}[Entropy decrement argument]  There exists a natural number $H$ between $H_-$ and $H_+$, which is a multiple of $a$, and such that
$$ \I( \X_H, \Y_H ) \leq \frac{H}{\log H \log\log\log H}.$$
\end{lemma}

As we shall see later, the key point here is that this bound is not only better than the trivial bound of $O(H)$ coming from \eqref{hyh}, but is (barely!) smaller than $H/\log H$ in the limit as $H \to \infty$; in particular, the mutual information between $\X_H$ and $\Y_H$ is smaller than the number $|{\mathcal P}_H|$ of primes one is using to define $F(\X_H,\Y_H)$.  One may think of this lemma as providing a weak independence between $\X_H$ and $\Y_H$ for certain large $H$.  For the purposes of optimising the bounds, it appears to be slightly more efficient to prove a variant of this lemma in which the right-hand side is of the form $\eps^{10} \frac{H}{\log H}$ (say); we leave the details to the interested reader.

\begin{proof} Suppose for sake of contradiction that one has
$$ \I( \X_H, \Y_H ) > \frac{H}{\log H \log\log\log H}$$
for all $H_- \leq H \leq H_+$ that are multiples of $a$.   Let $C_0$ be a sufficiently large natural number depending on $H_-$, and let $J$ be a sufficiently large natural number depending on $C_0,H_-,\eps$.  We may assume that $H_+$ is sufficiently large depending on $H_-, C_0, J$.  The idea is to now repeatedly use \eqref{splat} to decrement the entropy ratio $\frac{\HH(\X_H)}{H}$ as $H$ increases, until one arrives at the absurd situation of a random variable with negative entropy.

Let us recursively define the natural numbers $H_- \leq H_1 \leq H_2 \leq \dots \leq H_J$ by setting $H_1 := a H_-$ and
$$ H_{j+1} := H_j \lfloor C_0 \log H_j \log\log\log H_j \rfloor $$
for all $1 \leq j < J$.  Note that if $H_+$ is sufficiently large depending on $H_-,C_0,J$, then all the $H_j$ will lie between $H_-$ and $H_+$ and are multiples of $a$.  For $C_0$ large enough, we see from \eqref{splat} with $H,k$ replaced by $H_j$ and $\lfloor C_0 \log H_j \log\log\log H_j \rfloor$ respectively, followed by \eqref{jens}, that
$$\frac{\HH(\X_{H_{j+1}})}{H_{j+1}} \leq  \frac{\HH(\X_{H_j})}{H_j} - \frac{1}{2 \log H_j \log\log\log H_j}$$
for all $1 \leq j < J$.  (The $o_{A \to \infty}(1)$ error may be absorbed as we are assuming $A$ to be large.)
On the other hand, an easy induction\footnote{Alternatively, one can proceed by noting that for any given $T \geq H_-$, there are $\gg \frac{\log T}{\log\log T}$ values of $H_j$ between $T$ and $T^2$ if $J$ is large enough, which is sufficient to get some divergence in $\sum_{j=1}^J \frac{1}{2 \log H_j \log\log\log H_j}$ as $J \to \infty$.} 
 shows that there exists $B \geq 10^{10}$ (depending on $C_0,H_-$) such that
$$ H_j \leq \exp( B j \log j ) $$
for all $2 \leq j \leq J$.  Thus we have
$$\frac{\HH(\X_{H_{j+1}})}{H_{j+1}} \leq  \frac{\HH(\X_{H_j})}{H_j} - \frac{1}{2B j \log j \log\log(Bj\log j)}$$
for all $2 \leq j \leq J$, which on telescoping using \eqref{bad} gives the bound
$$ \sum_{j=2}^J \frac{1}{2Bj \log j \log\log(Bj\log j)} \ll_\eps 1.$$
But the sum on the left-hand side diverges (very slowly!) in the limit $J \to \infty$, and so we obtain a contradiction by choosing $J$ (and then $H_+$) large enough.
\end{proof}

From the above lemma we can find an $H$ between $H_-$ and $H_+$ that is a multiple of $a$, such that
\begin{equation}\label{ixy}
 \I( \X_H, \Y_H ) = o_{H_- \to \infty}\left(\frac{H}{\log H}\right).
\end{equation}
Fix this value of $H$.  From \eqref{mutual} and \eqref{ixy} we have
$$ \sum_x \P( \X_H = x ) \left(\HH(\Y_H) - \HH(\Y_H|\X_H=x)\right) =  o_{H_- \to \infty}\left(\frac{H}{\log H}\right).$$
By \eqref{jens}, \eqref{hayah}, the summands are bounded below by $-o_{A \to \infty}(1)$.  Thus, if we call a value $x$ \emph{good} if one has
\begin{equation}\label{sab}
 \HH(\Y_H) - \HH(\Y_H|\X_H=x) =  o_{H_- \to \infty}\left(\frac{H}{\log H}\right),
\end{equation}
we see from Markov's inequality that the random variable $\X_H$ will attain a good value with probability $1-o_{H_- \to \infty}(1)$.

Informally, if $x$ is good, then $\Y_H$ remains somewhat uniformly distributed across $\Z/P_H\Z$ even after one conditions $\X_H$ to equal $x$, in the sense that this conditioned random variable cannot concentrate too much mass into a small region.  More precisely, we have

\begin{lemma}[Weak uniform distribution]\label{weak-unif}    Let $x$ be a good value.
Let $E_x$ be a subset of $\Z/P_H\Z$ (which can depend on $x$) of cardinality
$$ |E_x| \leq \exp\left( - \eps^7 \frac{H}{\log H} \right) P_H.$$
Then one has
$$ \P( \Y_H \in E_x | \X_H = x ) = o_{H_- \to \infty}( 1 ).$$
\end{lemma}

The quantity $\eps^7$ here could be replaced by any other function of $\eps$, but we use this particular choice to match with Lemma \ref{h} below.

\begin{proof}  Applying \eqref{haxy} (conditioned to the event $\X_H = x$) we have
\begin{align*}
 \HH( \Y_H | \X_H = x, 1_{E_x}(\Y_H) ) &= \HH( \Y_H | \X_H = x ) + \HH( 1_{E_x}(\Y_H) | \Y_H, \X_H=x ) \\
&\quad - \HH(1_{E_x}(\Y_H) |  \X_H = x) \\
&\geq \HH( \Y_H | \X_H = x ) - \HH(1_{E_x}(\Y_H) | \X_H = x).
\end{align*}
By \eqref{xy} (again conditioned to the event $\X_H = x$), the left-hand side may be expanded as
\begin{align*}
&\P( \Y_H \in E_x | \X_H = x ) \HH( \Y_H | \X_H=x, \Y_H \in E_x ) \\
&\quad+ \P( \Y_H \not \in E_x | \X_H = x ) \HH( \Y_H | \X_H = x, \Y_H \not \in E_x) 
\end{align*}
and thus by \eqref{sab}
\begin{align*}
&\P( \Y_H \in E_x | \X_H = x ) \HH( \Y_H | \X_H=x, \Y_H \in E_x ) \\
&\quad+ \P( \Y_H \not \in E_x | \X_H = x ) \HH( \Y_H | \X_H = x, \Y_H \not \in E_x) \\
&\quad\quad \geq \HH(\Y_H) -  \HH(1_{E_x}(\Y_H) | \X_H = x) - o_{H_- \to \infty}\left( \frac{H}{\log H} \right).
\end{align*}
By \eqref{jens}, $\HH(1_{E_x}(\Y_H) | \X_H = x)$ is bounded by $\log 2$ and so this term can be absorbed in the $ o_{H_- \to \infty}(H/\log H)$ error.
From \eqref{hyx} we have
$$ \HH( \Y_H | \X_H = x, \Y_H \not \in E_x) \leq \HH(\Y_H)$$
and hence
$$ \P( \Y_H \in E_x | \X_H = x ) \left(\HH(\Y_H) - \HH( \Y_H | \X_H=x, \Y_H \in E_x ) \right) \leq  
o_{H_- \to \infty}\left(\frac{H}{\log H}\right).$$
But from \eqref{jens} one has
$$ \HH( \Y_H | \X_H=x, \Y_H \in E_x ) \leq \log |E_x| \leq \log P_H - \eps^7 \frac{H}{\log H} $$
and the claim then follows from \eqref{hayah} (recalling that $H_-$ is large depending on $\eps$).
\end{proof}

\begin{remark} Lemma \ref{weak-unif} may also be derived from the data processing inequality
$$ D_{KL}( 1_{E_x}(\Y'_H) || 1_{E_x}(\Y_H) ) \leq D_{KL}( \Y'_H || \Y_H )$$
where $\Y'_H$ is the random variable $\Y_H$ conditioned to the event $\X_H = x$, and where $D_{KL}( \X|| \Y) :=  \sum_x \P( \X = x) \log \frac{\P(\X=x)}{\P(\Y=x)}$ denotes the Kullback-Leibler divergence; we leave the details of this alternate derivation to the interested reader.  (Thanks to Yihong Wu for this observation.)
\end{remark}

We can use this weak uniform distribution to show that $F(\X_H,\Y_H)$ concentrates as a function of $\Y_H$.  We first observe

\begin{lemma}[Hoeffding inequality]\label{h} Let $x$ lie in the range of $\X_H$.  Let $E_x$ denote the set of all $y \in \Z/P_H\Z$ such that
$$ \left|F(x,y) - \frac{1}{P_H} \sum_{y' \in \Z/P_H\Z} F(x,y')\right| \geq \eps^2 \frac{H}{\log H}.$$
Then
$$ |E_x| \leq \exp\left( - \eps^7 \frac{H}{\log H} \right ) P_H.$$
\end{lemma}

\begin{proof}  We interpret this inequality probabilistically.  Let $\y$ be drawn uniformly at random from $\Z/P_H\Z$, then our task is to show that
$$ \P\left( |F(x,\y) - \E F(x,\y)| \geq \eps^2 \frac{H}{\log H} \right) \leq \exp\left( -\eps^7 \frac{H}{\log H} \right ).$$
We can write 
$$ F(x,\y) = \sum_{p \in {\mathcal P}_H} F_p(x,\y)$$
where
\begin{equation}\label{fp-def}
 F_p(x,\y) := c_p \sum_{j: j,j+ph \in [1,H]} 1_{a\y+j = pb \ (ap)} x_{1,j} x_{2,j+ph}.
\end{equation}
Note that the only randomness in the quantity $F_p(x,\y)$ comes from the reduction $\y\ (p)$ of $\y$ modulo $p$.  Since $\y$ is uniformly distributed in $\Z/P_H\Z$, we see from the Chinese remainder theorem that the $\y\ (p)$ are uniformly distributed in $\Z/p\Z$ and are jointly independent in $p$.  As each $F_p(x,\y)$ is a deterministic function of $\y\ (p)$, we conclude that the $F_p(x, \y)$ are also jointly independent in $p$.  On the other hand, since all $p \in {\mathcal P}_H$ lie in the interval $\frac{\eps^2}{2} H \leq p \leq \eps^2 H$, we have the deterministic bound $|F_p(x,\y)|\leq C/\eps^2$ for some absolute constant $C$.  Applying the Hoeffding inequality \cite{hoeff}, we conclude that
$$ \P\left( |F(x,\y) - \E F(x,\y)| \geq \eps^2 \frac{H}{\log H} \right) \ll \exp\left( - \frac{2 (\eps^2 \frac{H}{\log H})^2}{(2C/\eps^2)^2 |P_H|} \right).$$
From the prime number theorem we have $|P_H| \ll \eps^2 \frac{H}{\log H}$, and the claim follows (as $\eps$ is small and $H$ is large).
\end{proof}

Combining this lemma with Lemma \ref{weak-unif}, we conclude that for any good $x$, one has
$$
\P\left( \left|F(x,\Y_H) - \frac{1}{P_H} \sum_{y \in \Z/P_H\Z} F(x,y)\right| \geq \eps^2 \frac{H}{\log H}\right ) = o_{H_- \to \infty}( 1 ). $$
By Fubini's theorem, and the fact that $\X_H$ is good with probability $1-o_{H_- \to \infty}(1)$, one thus has
$$ F(\X_H, \Y_H) = \frac{1}{P_H} \sum_{y \in \Z/P_H\Z} F(\X_H,y) + O\left( \eps^2 \frac{H}{\log H} \right) $$
with probability $1-o_{H_- \to \infty}(1)$.  On the other hand, from the triangle inequality, \eqref{faxy}, and the prime number theorem we have
$$ F(x,y) \ll \frac{H}{\log H}.$$
We can thus take expectations and conclude that
$$ \E F(\X_H, \Y_H) = \E \frac{1}{P_H} \sum_{y \in \Z/P_H\Z} F(\X_H,y) + O\left( \eps^2 \frac{H}{\log H} \right),$$
and hence by \eqref{efy-large} we have
\begin{equation}\label{epoh}
\left |\E \frac{1}{P_H} \sum_{y \in \Z/P_H\Z} F(\X_H,y)\right| \gg \eps \frac{H}{\log H}.
\end{equation}
The advantage here is that we have decoupled the $x$ and $y$ variables, and the $y$ average is now easy to compute.  Indeed, from the Chinese remainder theorem and \eqref{fp-def} we see that
$$ \frac{1}{P_H} \sum_{y \in \Z/P_H\Z} F_p(x,y) =
\frac{c_p}{p} \sum_{j: j,j+ph \in [1,H]} 1_{j = pb \ (a)} x_{1,j} x_{2,j+ph} $$
for any $x$ and any $p \in {\mathcal P}$, and on summing in ${\mathcal P}$ and inserting into \eqref{epoh}, we conclude that
$$
\left| \E \sum_{p \in {\mathcal P}_H} \frac{c_p}{p} \sum_{j: j,j+ph \in [1,H]} 1_{j = pb \ (a)} g_{1,\eps^2}(a\n+j) g_{2,\eps^2}(a\n+j+ph)\right | \gg \eps \frac{H}{\log H}.
$$
Since $g_i = g_{i,\eps^2} + O(\eps^2)$ and $g_i,g_{i,\eps^2} = O(1)$ for $i=1,2$, we we can replace $g_{i,\eps^2}$ by $g_i$ on the left-hand side at the cost of an error of $O( \eps^2 \sum_{p \in {\mathcal P}_H} \frac{H}{p} ) = O(\eps^2 \frac{H}{\log H})$.  We thus have
\begin{equation}\label{epha}
\left| \E \sum_{p \in {\mathcal P}_H} \frac{c_p}{p} \sum_{j: j,j+ph \in [1,H]} 1_{j = pb \ (a)} g_{1}(a\n+j) g_{2}(a\n+j+ph)\right | \gg \eps \frac{H}{\log H}.
\end{equation}

On the other hand, by using the Hardy-Littlewood circle method, we can obtain the following deterministic estimate for the expression inside the expectation.

\begin{lemma}[Circle method estimate]  Let $a,H$ be as above (in particular, $H$ is a multiple of $a$). For any $\alpha \in \R/\Z$, let $S_H(\alpha)$ denote the exponential sum
\begin{equation}\label{sth}
S_H(\alpha) := \sum_{p \in {\mathcal P}_H} \frac{c_p}{p} e(\alpha p)
\end{equation}
and let $\Xi_H$ denote the elements $\xi \in \Z/H\Z$ for which
$$
\left|S_H\left(-\frac{(b+h) \eta}{a} - \frac{h \xi}{H}\right)\right| \geq \frac{\eps^2}{\log H}
$$
for some $\eta \in \Z/a\Z$.  For $j=1,\dots,H$, let $x_{1,j},x_{2,j}$ be complex numbers bounded in magnitude by one.  Then
\begin{equation}\label{cap}
\begin{split}
&\sum_{p \in {\mathcal P}_H} \frac{c_p}{p} \sum_{j: j,j+ph \in [1,H]} 1_{j = pb \ (a)} x_{1,j} x_{2,j+ph} \\
&\quad \ll_{a,h} \frac{H}{\log H} \left( \eps^2 + 
\sum_{\xi \in \Xi_H} \frac{1}{H} \left|\sum_{j=1}^H x_{1,j} e( -j \xi/H )\right|\right).
\end{split}
\end{equation}
\end{lemma}

\begin{proof}
We extend $x_{1,j}$, $x_{2,j}$ periodically with period $H$.  If we remove the constraint that $j+ph \in [1,H]$, we incur an error of
$O( \sum_{p \in {\mathcal P}_H} \frac{1}{p} |ph| ) = O(|h|\eps^2 \frac{H}{\log H})$ which is acceptable.  Thus, viewing $j$ now as an element of $\Z/H\Z$, we may replace the left-hand side of \eqref{cap} by
\begin{equation}\label{G1}
\sum_{p \in {\mathcal P}_H} \frac{c_p}{p} \sum_{j \in \Z/H\Z} 1_{j = pb \ (a)} x_{1,j} x_{2,j+ph}.
\end{equation}

We perform a Fourier expansion
$$ x_{i,j} = \sum_{\xi \in \Z/H\Z} G_i(\xi) e( j \xi / H )$$
for $i=1,2$, where
$$ G_i(\xi) := \frac{1}{H} \sum_{j \in \Z/H\Z} x_{i,j} e(-j \xi / H).$$
We can thus expand \eqref{G1} as
$$ \sum_{\xi, \xi' \in \Z/H\Z} G_1(\xi) G_2(-\xi') \sum_{p \in {\mathcal P}_H} \frac{c_p}{p} \sum_{j \in \Z/H\Z} 1_{j = pb \ (a)} e\left( \frac{j \xi}{H} - \frac{(j+ph) \xi'}{H}\right ).$$
The inner sum vanishes unless $\xi' = \xi + \frac{H}{a} \eta$ for some $\eta \in \Z/a\Z$, in which case one has
$$ \sum_{j \in \Z/H\Z} 1_{j = pb \ (a)} e\left( \frac{j \xi}{H} - \frac{(j+ph) \xi'}{H}\right ) = \frac{H}{a} e\left( -\frac{p(b+h)\eta}{a} - \frac{ph\xi}{H} \right )$$
(recall that $H$ was chosen to be a multiple of $a$),
and thus by \eqref{sth} we can write \eqref{G1} as
$$ \frac{H}{a} \sum_{\eta \in \Z/a\Z} \sum_{\xi \in \Z/H\Z} G_1(\xi) G_2(-\xi- \frac{H}{a} \eta) S_H\left( -\frac{(b+h) \eta}{a} - \frac{h \xi}{H} \right).$$

From the Cauchy-Schwarz inequality followed by the Plancherel identity, one has
$$  \sum_{\xi \in \Z/H\Z} |G_1(\xi)| |G_2(-\xi - \frac{H}{a} \eta)| \ll 1,$$
so those $\xi \not \in \Xi_H$ give an acceptable contribution.  For the remaining $\xi$, we bound $G_2(-\xi - \frac{H}{a} \eta)$ crudely by $O(1)$ and $S_H(-\frac{(b+h) \eta}{a} - \frac{h \xi}{H} ) $ by $O( \frac{1}{\log H})$
and use the triangle inequality to obtain the claim.
\end{proof}

Combining this lemma with \eqref{epha}, we conclude that
$$
\sum_{\xi \in \Xi_H}
\E \frac{1}{H} \left|\sum_{j=1}^H g_1(a\n+j) e( -j \xi/H )\right| \gg_{a,h} \eps.$$
By \eqref{hq-a} we thus have
$$ \eps \ll_{a,h} o_{H_- \to \infty}( |\Xi_H| ).$$

To conclude the desired contradiction, it thus suffices (by taking $H_-$ large enough) to show

\begin{lemma}[Restriction theorem for the primes]  We have $|\Xi_H| \ll_{a,h,\eps} 1$.
\end{lemma}

\begin{proof}  We invoke \cite[Proposition 4.2]{gt-selberg} (with $p=4$, $F(n) := n$, and $N$ replaced by $aH$), which gives the bound
$$ \left( \sum_{b \in \Z/aH\Z} | \frac{1}{aH} \sum_{n=1}^H a_n \beta_R(n) e( - bn/aH ) |^4 \right)^{1/4} \ll \left( \frac{1}{aH} \sum_{n=1}^H |a_n|^2 \beta_R(n) \right)^{1/2}$$
for any sequence $a_n$, where $R := (aH)^{1/10}$ and $\beta_R$ is a certain non-negative weight constructed in \cite[Proposition 3.1]{gt-selberg}, whose only relevant properties here are that $\beta_R(n) \gg \log H$ when $n$ is a prime in ${\mathcal P}_H$.  Setting $a_n$ set equal to $\frac{c_p}{p \beta_R(p)}$ when $n$ is a prime in ${\mathcal P}_H$, and $a_n=0$ otherwise, we conclude that\footnote{As an alternative proof of this estimate, one can use standard Fourier-analytic manipulations to rewrite the left-hand side of \eqref{eft} as $aH \sum_{p_1,p_2,p_3,p_4 \in {\mathcal P}_H:p_1+p_2=p_3+p_4} \frac{c_{p_1} c_{p_2} \overline{c_{p_3}} \overline{c_{p_4}}}{p_1 p_2 p_3 p_4}$, which by the triangle inequality is bounded in magnitude by $O_a( \frac{1}{H^3} \sum_{p_1,p_2,p_3,p_4 \in {\mathcal P}_H: p_1+p_2=p_3+p_4} 1 )$.  The sum may be upper bounded using a standard upper bound sieve for the primes (e.g. the Selberg sieve) to be $O_\eps( H^3 / \log^4 H)$, giving \eqref{eft}.}
\begin{equation}\label{eft}
 \sum_{k \in \Z/aH\Z} \left|S_H\left(\frac{k}{aH}\right)\right|^4 \ll_{\eps,a} \frac{1}{\log^4 H}
\end{equation}
and thus by Markov's inequality we have $|S_H(\frac{k}{aH})| \geq \frac{\eps^2}{\log H}$ for at most $O_{\eps,a}(1)$ values of $k \in \Z/aH\Z$.  The claim follows.
\end{proof}

\begin{remark} In the special case $g_1 = g_2 = \lambda$ (or more generally when $g_2$ is the complex conjugate of $g_1$, we have $c_p=1$, and the exponential sum $S_H(\alpha)$ can then be handled by the Vinogradov estimates for exponential sums over primes (see e.g. \cite[\S 13.5]{ik}).  In that case, one can compute $\Xi_H$ fairly explicitly; it basically consists of those frequencies $\xi$ which are ``major arc'' in the sense that $\xi/H$ is close to a rational $a/q$ of bounded denominator $q$.  As remarked previously, this allows for a slight simplification in the arguments in that the exponential sum estimates in \cite[Lemma 2.2, Theorem 2.3]{mrt} can be replaced with the simpler estimate in \cite[Theorem A.1]{mrt}; also, the quantitative bounds in Theorem \ref{lach} should improve if one uses this approach.  However, for more general choices of $g_1,g_2$, the coefficients $c_p$ are essentially arbitrary unit phases, and the frequency set $\Xi_H$ need not be contained within major arcs.
\end{remark}

\section{Further remarks}\label{remarks}

It is natural to ask if the arguments can be extended to higher point correlations than the $k=2$ case, for instance to bound sums such as the three-point correlation
\begin{equation}\label{3pt}
\sum_{x/\omega < n \leq x} \frac{\lambda(n) \lambda(n+1) \lambda(n+2)}{n}.
\end{equation}
Most of the above arguments carry through to this case.  However, the ``bilinear'' left-hand side of \eqref{G1} will be replaced by a ``trilinear'' expression such as
$$
\left | \sum_{p \in {\mathcal P}_H} \frac{1}{p} \sum_{j \in \Z/H\Z} x_{1,j} x_{2,j+p} x_{3,j+2p} \right |.$$
These sorts of sums have been studied in the ergodic theory literature \cite{fhk}, \cite{wz}.  Roughly speaking, the analysis there shows that these sums are small unless one has a large Fourier coefficient $G_1(\xi)$ for some $\xi \in \Z/H\Z$.  However, in contrast to the previous argument in which $\xi$ was restricted to a small set $\Xi_H$ (which, crucially, was independent of $\n$), one now has no control whatsoever on the location of $\xi$.  As such, one would now need to control maximal averaged exponential sums such as
\begin{equation}\label{xnil}
 \frac{1}{X} \int_X^{2X} \sup_\alpha \left|\frac{1}{H} \sum_{x \leq n \leq x+H} \lambda(n) e(\alpha n)\right|\ dx,
\end{equation}
which (as pointed out in \cite{mrt}) are not currently covered by the existing literature (note carefully that the supremum in $\alpha$ is \emph{inside} the integral over $x$).  However, this appears to be the only significant obstacle to extending the results of this paper to the $k=3$ case, and so it would certainly be of interest to obtain non-trivial estimates on \eqref{xnil}.   Note however that if one replaces $\lambda(n)$ with $n^{it}$, then the expression \eqref{xnil} exhibits essentially no cancellation for $t$ almost as large as $X^2$ (as opposed to the condition $t=O(X)$ that naturally appears in the $k=2$ analysis).  Similarly for the variant
$$
\sum_{x/\omega < n \leq x} \frac{n^{it} (n+1)^{-2it} (n+2)^{it}}{n} $$
of \eqref{3pt}.  This suggests that in order to establish cancellation in \eqref{3pt} and \eqref{xnil}, one must somehow go beyond the techniques in \cite{mr}, \cite{mrt}, as these techniques do not exclude the problematic multiplicative functions $n \mapsto n^{it}$ for $t$ between $x$ and $x^2$.

For even higher values of $k$, one has to now control quartilinear and higher expressions in place of \eqref{G1}.  Using the literature from higher order Fourier analysis (in particular the inverse theorem in \cite{gtz}, together with transference arguments from \cite{fhk}, \cite{gt-primes}, or \cite{wz}), one is now faced with the task of controlling sums even more complicated than \eqref{xnil}, in which the linear phases $n \mapsto e(\alpha n)$ are now replaced by more general nilsequences of higher step (which one then has to take the supremum over, before performing the integral); this task can be viewed as a local version of the machinery in \cite{fhost}, \cite{FH}, and will be carried out in detail in \cite{tao-detail}.  Of course, since satisfactory control on \eqref{xnil} is not yet available (even if one inserts logarithmic averaging), it is not feasible at present to control higher step analogues of \eqref{xnil} either.  However, one can hope that if a technique is found to give good bounds on \eqref{xnil}, it could also extend (in principle at least) to higher step sums.

It is of course of interest to remove the logarithmic averaging from Theorem \ref{lach} or Theorem \ref{elliott}.  It appears difficult to do this while utilising the entropy decrement argument, because this argument involves a scale $H$ which cannot be specified in advance, but is produced through a variant of the pigeonhole principle.  However, it may be possible to estimate expressions such as \eqref{pnj} for a specified $H$ without resorting to the entropy decrement argument, by establishing some sort of expander graph property for the random graph $G_{n,H}$ (or some closely related graph) from the introduction, and then there would be some chance of removing the logarithmic averaging.  Unfortunately we were unable to establish such an expansion property, as the edges in the graph $G_{n,H}$ do not seem to be either random enough or structured enough for standard methods of establishing expansion to work.

\end{document}